\documentclass[a4paper,11pt]{amsproc}
\usepackage{amssymb}
\usepackage{amsmath}
\usepackage{amsthm,lmodern} 
\usepackage{tikzsymbols}
\usepackage[foot]{amsaddr}

\usepackage{framed, booktabs, float}
\usepackage{xcolor}
\usepackage{enumerate}
\usepackage{mathtools}
\usepackage{url}
\usepackage[T1]{fontenc}
\usepackage{geometry}
\usepackage{tikz}
\usetikzlibrary{3d}
\usetikzlibrary{calc}

\usepackage{comment}

\theoremstyle{plain}
\newtheorem{Theorem} {Theorem} [section]

\newtheorem{Lemma} [Theorem] {Lemma}
\newtheorem{Corollary} [Theorem] {Corollary}

\newtheorem{Question} [Theorem] {Question}

\theoremstyle{definition}

\title{Multicolor vector space Ramsey numbers over the binary field}

\author[Bishnoi]{Anurag Bishnoi$^1$}
\address{$^1$Delft Institute of Applied Mathematics, Delft University of Technology, Netherlands.}
\email{A.Bishnoi@tudelft.nl}

\author[Kucheriya]{Gaurav Kucheriya$^2$}
\address{$^2$Department of Applied Mathematics, Charles University, Czechia.}
\email{gaurav@kam.mff.cuni.cz}
\thanks{Anurag Bishnoi is supported by the Dutch Research Council (NWO)
through the NWO Talent Programme Vidi, project ``Extremal problems in finite geometry''
(project number VI.Vidi.243.039).}
\thanks{Gaurav Kucheriya is supported by GAČR grant 25-17377S and GAUK project 92125.}

\begin{document}

\begin{abstract}
    For every fixed integer $t \geq 2$, we give an upper bound on the multicolor vector space Ramsey number $R_2(t; k)$ that is a tower function of height independent of $k$. 
    For $t \geq 3$, this is the first bound of its form, significantly improving upon the earlier bounds that are towers of height linear in $k$. 
    We achieve this by reducing the problem to a classical hypergraph Ramsey problem via binary simplex codes.
    In particular, we prove that 
    $$R_2(t; k) \leq \left\lceil \log R(K_s^{(r)}; k + 1) \right\rceil \leq \mathrm{twr}_{r-1}(c k\log k),$$ for $r = 2^{t - 1}$ and $s = 2^t - 1$, where $R(K_{s}^{(r)}; k + 1)$ is the classical $(k + 1)$-color Ramsey number for the complete $r$-uniform hypergraph on $s$ vertices. 
    This improvement also translates into an improved lower bound on the chromatic number of the binary projective space with respect to $(t - 1)$-flats.
    For $t = 2$, it recovers the connection with multicolor Ramsey numbers for triangles.
\end{abstract}

\maketitle

\section{Introduction}
We study the $q$-analogs of classical Ramsey numbers, called vector space Ramsey numbers over a finite field $\mathbb{F}_q$. 
In particular, we focus on $q = 2$, where these numbers can be defined as follows. 
Let $\mathbb{F}_2^n$ denote the $n$-dimensional binary vector space, and let $\mathrm{PG}(n - 1, 2)$ denote the corresponding projective space. 
We can identify the points of $\mathrm{PG}(n - 1, 2)$ using the nonzero elements of $\mathbb{F}_2^n$. 
For any integer $t \geq 1$, the $(t - 1)$-dimensional projective subspace of $\mathrm{PG}(n - 1, 2)$ corresponds to the set of nonzero vectors in a $t$-dimensional vector subspace of $\mathbb{F}_2^n$. 

For integers $t, k \geq 1$, the (binary) multicolor vector space Ramsey number $R_2(t; k)$ is the least integer $n$ such that every $k$-coloring of the points of $\mathrm{PG}(n - 1, 2)$ contains a monochromatic $(t-1)$-dimensional projective subspace. 
The existence of these numbers, over arbitrary finite fields, follows from the vector space Ramsey theorem of Graham, Leeb, and Rothschild~\cite{GrahamLeebRothschild}, but we do not get any quantitative bounds from their proof. 
Other early proofs of this result give extremely large quantitative bounds because of the repeated use of the Hales--Jewett theorem~\cite{Shelah,Spencer}.
Over $\mathbb{F}_2$, one can do better by reducing the problem to Taylor's disjoint unions theorem \cite{Taylor}, which gives 
\[R_2(t; k) \leq \mathrm{twr}_{2k(t-1)}(k),\]
where the tower function $\mathrm{twr}$ is defined recursively as $\mathrm{twr}_1(x) = x$ and $\mathrm{twr}_i(x) = 2^{\mathrm{twr}_{i-1}(x)}$ for all $i \geq 2$, with the subscript $i$ denoting the height of the tower.

A recent result of Frederickson and Yepremyan \cite{FredericksonYepremyan} brings the height of this tower function down to $(k - 1)(t - 1) + 1$. 
In particular, for a fixed $t$, the best previously known bounds had tower functions of height linear in $k$. 
In this paper, we show that the dependence on $k$ in tower height can be completely removed. 

\begin{Theorem}
\label{thm:main1}
    For every $t \geq 2$, there exists a constant $c_t > 0$ such that, for every $k \geq 2$
    \[R_2(t; k) \leq \mathrm{twr}_{r - 1}(c_t k \log k),\]
    where $r = 2^{t - 1}$. 
\end{Theorem}

Our new upper bound follows from the relation that we prove between the binary vector space Ramsey numbers and the classical hypergraph Ramsey numbers \cite{MubayiSukSurvey}. 
For an $r$-uniform hypergraph $H$, we denote by $R(H; k)$ the smallest $n$ such that every $k$-coloring of the $r$-element subsets of $[n]$ contains a monochromatic copy of $H$. 
We denote the complete $r$-uniform hypergraph on $s > r$ vertices by $K_s^{(r)}$.
Throughout, $\log$ denotes the base-two logarithm.

\begin{Lemma}
    For every $k, t \geq 2$, 
    \[R_2(t; k) \leq \left\lceil \log R(K_s^{(r)}; k + 1) \right\rceil,\]
    with $s = 2^t - 1$ and $r = 2^{t - 1}$.
\end{Lemma}

For the special case of $t = 2$, this lemma is well-known, as discussed in \cite{BishnoiCamesVanBatenburgRavi}. 
It follows from the classical argument that relates Schur numbers to the multicolor Ramsey number for triangles, when adapted to the abelian group $\mathbb{F}_2^n$ (see \cite{AbbottHanson1972} for a more general statement over arbitrary abelian groups). 
Our main contribution is to generalize that argument to $t > 2$ using multicolor Ramsey numbers of hypergraphs.

The upper bound on $R_2(t; k)$ can equivalently be stated as a lower bound on the chromatic number $\chi_2(t; n)$ \cite{BishnoiCamesVanBatenburgRavi}, which is the smallest number of colors needed to color the points of $\mathrm{PG}(n - 1, 2)$ so that there are no monochromatic $(t - 1)$-flats. 
From this definition it follows that $\chi_2(t; n) > k$ if and only if $R_2(t; k) \leq n$. 
By asymptotically inverting the bound on $R_2(t; k)$ from \cite{FredericksonYepremyan}, the previously best-known lower bound on $\chi_2(t; n)$ is $\Omega(\log^*n)$, the iterated logarithm of $n$. In contrast, by applying Theorem~\ref{thm:main1}, with $k$ of order $\log^{(r - 2)} n/\log^{(r - 1)} n$, we obtain the following significant improvement. 

\begin{Theorem}
    \label{thm:main2}
    For every $t \geq 2$, there is a constant $c'_t > 0$ such that, for all sufficiently large $n$,
    \[
        \chi_2(t; n) \geq c'_t
        \frac{\log^{(r-2)} n}{\log^{(r-1)} n},
    \]
    where $r = 2^{t - 1}$, and where $\log^{(0)} n = n$ and
    $\log^{(i)} n = \log(\log^{(i-1)} n)$ for $i \geq 1$.
\end{Theorem}

In Section~\ref{sec2}, we give some 
coding-theoretic background that will be used in the proof. 
In Section~\ref{sec3}, we prove our main result. 
We conclude with some future directions and open problems in Section~\ref{sec4}.

\section{Background and a useful lemma} \label{sec2}

Let $q$ be a prime power. The \emph{$q$-ary simplex code of dimension $t$}, denoted by $\mathcal{S}_{q,t}$, is the linear code over $\mathbb{F}_q$ whose generator matrix has as columns one representative from each one-dimensional subspace of $\mathbb{F}_q^t$. 
Equivalently, the columns are indexed by the points of the projective space $\mathrm{PG}(t-1,q)$. Since \begin{equation*} 
|\mathrm{PG}(t-1,q)|=\frac{q^t-1}{q-1}, 
\end{equation*} the code $\mathcal{S}_{q,t}$ has length $(q^t-1)/(q-1)$ and dimension $t$. 
This construction and the following standard constant-weight property
may be found in \cite[Theorem~2.7.5, p.~82]{HuffmanPless}.

\begin{Lemma} 
\label{lem:simplex}
    The $q$-ary simplex code $\mathcal{S}_{q,t}$ has parameters 
    \begin{equation*} \left[\frac{q^t-1}{q-1},\,t,\,q^{t-1}\right]_q. \end{equation*} 
    Moreover, every nonzero codeword of $\mathcal{S}_{q,t}$ has Hamming weight $q^{t-1}$. 
\end{Lemma}

We note a simple linear algebraic lemma which will be combined with the simplex codes to prove our main result. 

\begin{Lemma}\label{lem:isomorphism}
    Let $C \leq \mathbb{F}_q^m$ be a $t$-dimensional vector subspace. 
    Let $x_1, \dots, x_m$ be vectors in $\mathbb{F}_q^n$. 
    Define
    \[
        S \coloneqq \left\{\sum_{i \in [m]} c_i x_i :
        (c_1, \dots, c_m) \in C\right\}.
    \]
    If, for every $(c_1,\dots,c_m)\in C$,
    \[
        \sum_{i \in [m]} c_i x_i = \vec{0}
        \implies c_i = 0 \text{ for all } i,
    \]
    then $S \cong C$.
\end{Lemma}

\begin{proof}
    Consider the map $\phi : C \to S$ defined by
    \[
        \phi(c_1,\ldots,c_m)
        \coloneqq \sum_{i=1}^{m} c_i x_i.
    \]
    Note that $\phi$ is linear since $\phi(\vec{0})=\vec{0}$ and
    \[
        \phi(c+\lambda c')
        = \sum_{i=1}^{m} (c_i+\lambda c'_i)x_i
        = \sum_{i=1}^{m} c_i x_i
        + \lambda\sum_{i=1}^{m} c'_i x_i
        = \phi(c)+\lambda\phi(c'),
    \]
    for $c=(c_1,\dots,c_m),c'=(c'_1,\dots,c'_m)\in C$ and
    $\lambda\in\mathbb{F}_q$.

    Since  $\ker \phi = \{\vec{0}\}$ by assumption,
    $\phi$ is injective. 
    Moreover, $\phi$ is surjective
    by the definition of $S$. Therefore, $\phi$ is a vector space isomorphism
    and hence $S \cong C$.
\end{proof}

\section{Main result}\label{sec3}
Our proof idea for Theorem~\ref{thm:main1} is to turn a coloring of $\mathrm{PG}(n - 1,2)$ that avoids monochromatic $(t - 1)$-flats to a coloring of the complete $r$-uniform hypergraph on $\mathbb{F}_2^n$ that avoids monochromatic $K_s^{(r)}$, for carefully chosen $r$ and $s$. 
For each $r$-element subset $\{x_1, \dots, x_r\}$ we consider the vector $x_1 + \dots + x_r$.
If this vector is nonzero, then it corresponds to a projective point and we give it a color according to the coloring of $\mathrm{PG}(n - 1, 2)$.
If it is the zero vector, then we give it a special new color. 
We then prove the absence of a monochromatic $K_s^{(r)}$ by using Lemmas~\ref{lem:simplex} and \ref{lem:isomorphism} for $q = 2$.

\begin{Theorem}
    If $\chi_2(t; n) \leq k$, then $R(K_s^{(r)}; k + 1) > 2^n$, where $r = 2^{t - 1}$ and $s = 2^t - 1$.
\end{Theorem}
\begin{proof}
    Let $\chi$ be a proper $k$-coloring of $\mathrm{PG}(n - 1, 2)$, that is, there are no monochromatic $(t - 1)$-flats under $\chi$. 
    Let $r = 2^{t - 1}$ and $s = 2^t - 1$.
    Define a coloring $c$ on the $r$-element subsets of $V = \mathbb{F}_2^n$ by 
    \begin{equation*}
        c(\{x_1, \dots, x_r\}) = 
        \begin{cases}
            \star, & \text{if } x_1 + \cdots + x_r = 0, \\
            \chi(x_1+ \dots+ x_r), & \text{otherwise.}
        \end{cases}
    \end{equation*}
    We will show that there is no monochromatic $K_s^{(r)}$ in this coloring of the complete $r$-uniform hypergraph on $V$, which proves that $R(K_s^{(r)}; k+1) > 2^n.$
   
    Let $x_1, \dots, x_s$ be a set of distinct vectors in $\mathbb{F}_2^n$.
    First, assume that they induce a monochromatic $K_s^{(r)}$ in color $\star$, that is, 
    $\sum_{i \in T} x_i = 0$ for all subsets $T$ of $[s]$ with $|T| = r$. 
    Since $r+1 \leq s$, we can take $T_1 = \{1, \dots, r\}$ and $T_2 = \{1, \dots, r - 1, r+1\}$ as two such subsets. 
    Then we get $\sum_{i \in T_1} x_i = \sum_{i \in T_2} x_i = 0$, which implies $x_r = x_{r + 1}$, a contradiction to the fact that we have distinct vectors. 

    Now assume that these vectors give a monochromatic $K_s^{(r)}$ in one of the $k$ colors of $\chi$. 
    In particular, $\sum_{i \in T} x_i \neq 0$ for all $T \subseteq [s]$ of size $r$. 
     
    Let $C$ be the binary simplex code $\mathcal{S}_{2,t}$.
    By Lemma~\ref{lem:simplex}, it has parameters $[2^{t} - 1, t, 2^{t-1}]_2$ and every nonzero codeword in $C$ has Hamming weight $2^{t - 1}$.
    Define $S = \left\{\sum_{i \in [s]}c_i x_i : (c_1, \dots, c_s) \in C \setminus \{\vec{0}\}\right\}$ and $\overline{S} = S \cup \{\vec{0}\}$. 
    The Hamming weight of each nonzero codeword in $C$ is $r = 2^{t - 1}$, and the sum of any $r$ vectors among $x_1, \dots, x_s$ is nonzero.
    By Lemma~\ref{lem:isomorphism}, for $q = 2$, we get that $\overline{S}$ is a $t$-dimensional vector subspace of $\mathbb{F}_2^n$, that is $S$ corresponds to a copy of $\mathrm{PG}(t - 1, 2)$ in $\mathrm{PG}(n - 1, 2)$.
    All elements of $S$ receive the same color under $\chi$, which is a contradiction to the fact that $\chi$ is a proper coloring. 
\end{proof}
We now prove Theorem~\ref{thm:main1}.
\begin{Corollary}
    $R_2(t; k) \leq \left\lceil \log R(K_s^{(r)}; k + 1) \right\rceil \leq \mathrm{twr}_{r-1}(c_t k\log k)$,  for $r = 2^{t - 1}$, $s = 2^t - 1$, and some constant $c_t$.
\end{Corollary}
\begin{proof}
    Let $n = R_2(t; k) - 1$.
    Then there exists a $k$-coloring of $\mathrm{PG}(n - 1, 2)$ without monochromatic $(t - 1)$-flats, that is, $\chi_2(t; n) \leq k$.
    Therefore, $2^n < R(K_s^{(r)}; k + 1)$, which implies 
    $R_2(t; k) = n + 1 < 1 + \log R(K_s^{(r)}; k + 1)$.

    From the classical upper bound of Erd\H{o}s and Rado
\cite[Theorem~1, pp.~420--423]{ErdosRado1952}, one can deduce the stated upper bound on $R(K_s^{(r)}; k + 1)$. 
In particular, from \cite[Theorem~1]{AxenovichGyarfasLiuMubayi2014} we see that 
for every fixed $s>r\geq 2$, there is a constant
$c=c(r,s)>0$ such that $R(K_s^{(r)}; k + 1) \leq \mathrm{twr}_r(c (k + 1) \log(k + 1))$. 
By taking $\log$ and adjusting the constant we get the stated bound. 
\end{proof}

\section{Conclusion}
\label{sec4}
To improve our bound on $R_2(t; k)$ one could try using a better code than the simplex code. 
However, we do need constant weight codes so that we have a uniform hypergraph. 
The simplex code achieves the minimum possible nonzero weight for a $t$-dimensional $q$-ary linear code of constant weight.

Another avenue for improvement is to bound $R_2(t; k)$ using the Ramsey number of the hypergraph whose edges are supported on the simplex code, along with one extra edge to take care of the special color. 
This is a much sparser hypergraph than the complete hypergraph, but we were unable to find better upper bounds.
It is an interesting open problem to explore the multicolor Ramsey number of these simplex hypergraphs. 

\begin{Question}
    Let $\mathcal{F}_t$ denote the $2^{t - 1}$-uniform hypergraph on $2^t - 1$ vertices whose edges are supported on the binary simplex code. 
    Determine the asymptotic growth of $R(\mathcal{F}_t; k)$ for all $t \geq 3$.
    In particular, does it have tower-type growth, and, if so, what is its tower height?
\end{Question}

Another natural question to ask is whether this argument generalizes to non-binary multicolor vector space Ramsey numbers. 
The coding-theoretic idea remains intact, but we were unable to translate the problem into a classical hypergraph Ramsey number for $q > 2$. 
We might need a new version of Ramsey numbers for hypergraphs whose edges are weighted by elements of $\mathbb{F}_q^\times$. 
We believe that our approach could yield upper bounds for all $q$.
Our arguments could also be useful for the more general vector space Ramsey numbers, where there has been tremendous progress in recent years \cite{FredericksonYepremyan, HunterPohoata}.

For $q>3$, the known general upper bounds remain enormous \cite{Shelah, Spencer}.
Perhaps $q = 3$ is a good target case since the current best upper bound on the multicolor case \cite[Theorem 1.3]{FredericksonYepremyan} is also a tower of height $(k - 1)(t - 1) + 1$. 
\begin{Question}
    Is $R_3(t; k) \leq \mathrm{twr}_r(k)$ for some constant $r = r(t)$?
\end{Question}

For general $q$ and $t$, the current best lower bound \cite{BishnoiCamesVanBatenburgRavi} is $R_q(t; k) \geq tk - O(1)$. 
Thus, there is much room for improvement.

\end{document}